\documentclass[11pt, a4paper]{article}

\usepackage{amsmath, amsthm, amssymb, url}
\usepackage[margin=1in]{geometry}
\usepackage{multirow}
\usepackage{tikz, tkz-graph, mathrsfs, enumerate}

\usetikzlibrary{arrows}
\newcommand{\vertex}[3]{\node [vertex] (#1) at (#2, #3 * 1.7) {};}

\newcommand{\Tran}{\mathrm{Tran}}
\newcommand{\Sym}{\mathrm{Sym}}

\newcommand{\id}{\mathrm{id}}
\newcommand{\Char}{\mathrm{char}}
\newcommand{\CA}{\mathrm{CA}}
\newcommand{\ICA}{\mathrm{ICA}}
\newcommand{\LCA}{\mathrm{LCA}}

\newcommand{\End}{\mathrm{End}}
\newcommand{\Sing}{\mathrm{Sing}}

\theoremstyle{plain}

\newtheorem{corollary}{Corollary}
\newtheorem{lemma}{Lemma}

\newtheorem{theorem}{Theorem}
\newtheorem*{claim*}{Claim}

\theoremstyle{definition}
\newtheorem{definition}{Definition}

\newtheorem{example}{Example}
\newtheorem{remark}{Remark}

\begin{document}

\title{Von Neumann Regular Cellular Automata}
\author{Alonso Castillo-Ramirez and Maximilien Gadouleau}

\newcommand{\Addresses}{{
  \bigskip
  \footnotesize

A. Castillo-Ramirez (Corresponding author), \textsc{Universidad de Guadalajara, CUCEI, Departamento de Matem\'aticas, Guadalajara, M\'exico.} \par \nopagebreak
Email: \texttt{alonso.castillor@academicos.udg.mx}

\medskip

M. Gadouleau, \textsc{School of Engineering and Computing Sciences, Durham University, South Road, Durham DH1 3LE, U.K.} \par \nopagebreak
 Email: \texttt{m.r.gadouleau@durham.ac.uk}
}}

\maketitle

\begin{abstract}
For any group $G$ and any set $A$, a cellular automaton (CA) is a transformation of the configuration space $A^G$ defined via a finite memory set and a local function. Let $\CA(G;A)$ be the monoid of all CA over $A^G$. In this paper, we investigate a generalisation of the inverse of a CA from the semigroup-theoretic perspective. An element $\tau \in \CA(G;A)$ is \emph{von Neumann regular} (or simply \emph{regular}) if there exists $\sigma \in \CA(G;A)$ such that $\tau \circ \sigma \circ \tau = \tau$ and $\sigma \circ \tau \circ \sigma = \sigma$, where $\circ$ is the composition of functions. Such an element $\sigma$ is called a \emph{generalised inverse} of $\tau$. The monoid $\CA(G;A)$ itself is regular if all its elements are regular. We establish that $\CA(G;A)$ is regular if and only if $\vert G \vert = 1$ or $\vert A \vert = 1$, and we characterise all regular elements in $\CA(G;A)$ when $G$ and $A$ are both finite. Furthermore, we study regular linear CA when $A= V$ is a vector space over a field $\mathbb{F}$; in particular, we show that every regular linear CA is invertible when $G$ is torsion-free elementary amenable (e.g. when $G=\mathbb{Z}^d, \ d \in \mathbb{N}$) and $V= \mathbb{F}$, and that every linear CA is regular when $V$ is finite-dimensional and $G$ is locally finite with $\Char(\mathbb{F}) \nmid o(g)$ for all $g \in G$.    
\end{abstract}


\section{Introduction} \label{sec:introduction}

Cellular automata (CA), introduced by John von Neumann and Stanislaw Ulam in the 1940s, are models of computation with important applications to computer science, physics, and theoretical biology. We follow the modern general setting for CA presented in \cite{CSC10}. For any group $G$ and any set $A$, a CA over $G$ and $A$ is a transformation of the configuration space $A^G$ defined via a finite memory set and a local function. Most of the classical literature on CA focus on the case when $G=\mathbb{Z}^d$, for $d\geq1$, and $A$ is a finite set (see \cite{Ka05}), but important results have been obtained for larger classes of groups (e.g., see \cite{CSC10} and references therein).

Recall that a \emph{semigroup} is a set equipped with an associative binary operation, and that a \emph{monoid} is a semigroup with an identity element. Let $\CA(G;A)$ be the set of all CA over $G$ and $A$. It turns out that, equipped with the composition of functions, $\CA(G;A)$ is a monoid. In this paper we apply functions on the right; hence, for $\tau,\sigma \in \CA(G;A)$, the composition $\tau \circ \sigma$, denoted simply by $\tau \sigma$, means applying first $\tau$ and then $\sigma$. 

In general, $\tau \in \CA(G;A)$ is \emph{invertible}, or \emph{reversible}, or \emph{a unit}, if there exists $\sigma \in \CA(G;A)$ such that $\tau \sigma = \sigma \tau = \id$. In such case, $\sigma$ is called \emph{the inverse} of $\tau$ and denoted by $\sigma = \tau^{-1}$. When $A$ is finite, it may be shown that $\tau \in \CA(G;A)$ is invertible if and only if it is a bijective function (see \cite[Theorem 1.10.2]{CSC10}).

We shall consider the notion of \emph{regularity} which, coincidentally, was introduced by John von Neumann in the context of rings, and has been widely studied in semigroup theory (recall that the multiplicative structure of a ring is precisely a semigroup). Intuitively, cellular automaton $\tau \in \CA(G;A)$ is \emph{von Neumann regular} if there exists $\sigma \in \CA(G;A)$ mapping any configuration in the image of $\tau$ to one of its preimages under $\tau$. Clearly, this generalises the notion of reversibility.  

Henceforth, we use the term `regular' to mean `von Neumann regular'. Let $S$ be any semigroup. For $a, b \in S$, we say that $b$ is \emph{a weak generalised inverse} of $a$ if 
\[ aba=a. \]
We say that $b$ is \emph{a generalised inverse} (often just called \emph{an inverse}) of $a$ if
\[ aba = a \text{ and } bab= b. \]
An element $a \in S$ may have none, one, or more (weak) generalised inverses. It is clear that any generalised inverse of $a$ is also a weak generalised inverse; not so obvious is that, given the set $W(a)$ of weak generalised inverses of $a$ we may obtain the set $V(a)$ of generalised inverses of $a$ as follows (see \cite[Exercise 1.9.7]{CP61}):
\[ V(a) = \{ b a b^\prime : b, b^\prime \in W(a)  \}. \]
An element $a \in S$ is \emph{regular} if it has at least one generalised inverse (which is equivalent of having at least one weak generalised inverse). A semigroup $S$ itself is called \emph{regular} if all its elements are regular. Many of the well-known types of semigroups are regular, such as idempotent semigroups (or \emph{bands}), full transformation semigroups, and Rees matrix semigroups. Among various advantages, regular semigroups have a particularly manageable structure which may be studied using the so-called Green's relations. For further basic results on regular semigroups see \cite[Section 1.9]{CP61}. 

Another generalisation of reversible CA has appeared in the literature before \cite{ZZ12,ZZ15} using the concept of \emph{Drazin inverse} \cite{D58}. However, as Drazin invertible elements are a special kind of regular elements, our approach turns out to be more general and natural.  

In the following sections we study the regular elements in monoids of CA. First, in Section \ref{regular} we present some basic results and examples, and we establish that, except for the trivial cases $\vert G \vert = 1$ and $\vert A \vert = 1$, the monoid $\CA(G;A)$ is not regular. In Section \ref{finite}, we study the regular elements of $\CA(G;A)$ when $G$ and $A$ are both finite; in particular, we characterise them and describe a regular submonoid. In Section \ref{linear}, we study the regular elements of the monoid $\LCA(G; V)$ of linear CA, when $V$ is a vector space over a field $\mathbb{F}$. Specifically, using results on group rings, we show that, when $G$ is torsion-free elementary amenable (e.g., $G=\mathbb{Z}^d$), $\tau \in \LCA(G; \mathbb{F})$ is regular if and only if it is invertible, and that, for finite-dimensional $V$, $\LCA(G; V)$ itself is regular if and only if $G$ is locally finite and $\Char(\mathbb{F}) \nmid \vert \langle g \rangle \vert$, for all $g \in G$. Finally, for the particular case when $G \cong \mathbb{Z}_n$ is a cyclic group, $V := \mathbb{F}$ is a finite field, and $\Char(\mathbb{F}) \mid n$, we count the total number of regular elements in $\LCA(\mathbb{Z}_n ;\mathbb{F})$.


\section{Regular cellular automata} \label{regular}

For any set $X$, let $\Tran(X)$, $\Sym(X)$, and $\Sing(X)$, be the sets of all functions, all bijective functions, and all non-bijective (or singular) functions of the form $\tau : X \to X$, respectively. Equipped with the composition of functions, $\Tran(X)$ is known as the \emph{full transformation monoid} on $X$, $\Sym(X)$ is the \emph{symmetric group} on $X$, and $\Sing(X)$ is the \emph{singular transformation semigroup} on $X$. When $X$ is a finite set of size $\alpha$, we simply write $\Tran_\alpha$, $\Sym_\alpha$, and $\Sing_\alpha$, in each case.

We shall review the broad definition of CA that appears in \cite[Sec.~1.4]{CSC10}. Let $G$ be a group and $A$ a set. Denote by $A^G$ the \emph{configuration space}, i.e. the set of all functions of the form $x:G \to A$. For each $g \in G$, denote by $R_g : G \to G$ the right multiplication function, i.e. $(h)R_g := hg$ for any $h \in G$. We emphasise that we apply functions on the right, while \cite{CSC10} applies functions on the left.   

\begin{definition} \label{def:ca}
Let $G$ be a group and $A$ a set. A \emph{cellular automaton} over $G$ and $A$ is a transformation $\tau : A^G \to A^G$ satisfying the following: there is a finite subset $S \subseteq G$, called a \emph{memory set} of $\tau$, and a \emph{local function} $\mu : A^S \to A$ such that
\[ (g)(x)\tau = (( R_g \circ x  )\vert_{S}) \mu, \ \forall x \in A^G, g \in G,  \]
where $(R_g \circ x )\vert_{S}$ is the restriction to $S$ of $(R_g \circ x) : G \to A$ .
\end{definition} 

The group $G$ acts on the configuration space $A^G$ as follows: for each $g \in G$ and $x \in A^G$, the configuration $x \cdot g \in A^G$ is defined by 
\[ (h)x \cdot g := (hg^{-1})x, \quad \forall h \in G. \]
A transformation $\tau : A^G \to A^G$ is \emph{$G$-equivariant} if, for all $x \in A^G$, $g \in G$,
\[ (x \cdot g) \tau = ( (x) \tau ) \cdot g .\] 
Any cellular automaton is $G$-equivariant, but the converse is not true in general. A generalisation of Curtis-Hedlund Theorem (see \cite[Theorem 1.8.1]{CSC10}) establishes that, when $A$ is finite, $\tau : A^G \to A^G$ is a CA if and only if $\tau$ is $G$-equivariant and continuous in the prodiscrete topology of $A^G$; in particular, when $G$ and $A$ are both finite, $G$-equivariance completely characterises CA over $G$ and $A$. 

A configuration $x \in A^G$ is called \emph{constant} if $(g)x = k$, for a fixed $k \in A$, for all $g \in G$. In such case, we denote $x$ by $\mathbf{k} \in A^G$. 

\begin{remark}\label{constant}
It follows by $G$-equivariance that any $\tau \in \CA(G;A)$ maps constant configurations to constant configurations.  
\end{remark}

Recall from Section \ref{sec:introduction} that $\tau \in \CA(G;A)$ is \emph{invertible} if there exists $\sigma \in \CA(G;A)$ such that $\tau \sigma = \sigma \tau = \id$, and that $\tau \in \CA(G;A)$ is \emph{regular} if there exists $\sigma \in \CA(G; A)$ such that $\tau \sigma \tau = \tau$. We now present some examples of CA that are regular but not invertible.

\begin{example}
Let $G$ be any nontrivial group and $A$ any set with at least two elements. Let $\sigma \in \CA(G;A)$ be a CA with memory set $\{s \} \subseteq G$ and local function $\mu : A \to A$ that is non-bijective. Clearly, $\sigma$ is not invertible. As $\Sing(A)$ is a regular semigroup (see \cite[Theorem II]{H66}), there exists $\mu^\prime : A \to A$ such that $\mu \mu^\prime \mu = \mu$. If $\sigma^\prime$ is the CA with memory set $\{ s^{-1} \}$ and local function $\mu^\prime$, then $\sigma \sigma^\prime \sigma = \sigma$. Hence $\sigma$ is regular.
\end{example}

\begin{example}
Suppose that $A = \{0,1, \dots, q-1\}$, with $q \geq 2$. Consider $\tau_1, \tau_2 \in \CA(\mathbb{Z};A)$ with memory set $S:=\{-1,0,1 \}$ and local functions
\[ (x)\mu_1 = \min \{  (-1)x, (0)x, (1)x \} \text{ and } (x)\mu_2 = \max \{ (-1)x, (0)x ,  (1)x \}, \]
respectively, for all $x \in A^S$. Clearly, $\tau_1$ and $\tau_2$ are not invertible, but we show that they are generalised inverses of each other, i.e.  $\tau_1 \tau_2 \tau_1 = \tau_1$ and $\tau_2 \tau_1 \tau_2 = \tau_2$, so they are both regular. We prove only the first of the previous identities, as the second one is symmetrical. Let $x \in A^\mathbb{Z}$, $y := (x)\tau_1$, $z:= (y)\tau_2$, and $a := (z)\tau_1$. We want to show that $y = a$. For all $i \in \mathbb{Z}$ and $\epsilon \in \{ -1,0,1 \}$, we have
\[ (i + \epsilon )y = \min\{ (i + \epsilon - 1) x, (i + \epsilon)x, (i+\epsilon + 1)x \} \leq (i)x.\]
Hence,
\[(i)z = \max\{ (i-1)y, (i)y , (i+1)y \} \leq (i)x. \]
Similarly $(i-1)z \leq (i-1)x$ and $(i+1)z \leq (i+1)x$, so
\[ (i)a = \min\{ (i-1)z, (i)z , (i+1)z  \} \le (i)y = \min \{ (i-1)x, (i)x, (i+1)x \}. \]
Conversely, we have $(i-1)z, (i)z, (i+1)z \ge (i)y$, so $(i)a \ge (i)y$.
In particular, when $q=2$, $\tau_1$ and $\tau_2$ are the elementary CA known as Rules 128 and 254, respectively.
\end{example}

The following lemma gives an equivalent definition of regular CA. Note that this result still holds if we replace $\CA(G;A)$ with any monoid of transformations. 

\begin{lemma}\label{le-regular}
Let $G$ be a group and $A$ a set. Then, $\tau \in \CA(G;A)$ is regular if and only if there exists $\sigma \in \CA(G;A)$ such that for every $y \in (A^G) \tau$ there is $\hat{y} \in A^G$ with $(\hat{y})\tau = y$ and $(y)\sigma = \hat{y}$.
\end{lemma}
\begin{proof}
If $\tau \in \CA(G;A)$ is regular, there exists $\sigma \in \CA(G;A)$ such that $\tau \sigma \tau = \tau$. Let $x \in A^G$ be such that $(x)\tau = y$ (which exists because $y \in (A^G) \tau$) and define $\hat{y}:= (y)\sigma$. Now,
\[ (\hat{y})\tau = (y)\sigma\tau = (x)\tau\sigma\tau = (x)\tau = y. \]
Conversely, assume there exists $\sigma \in \CA(G;A)$ satisfying the statement of the lemma. Then, for any $x \in A^G$ with $y:=(x)\tau$ we have
\[ (x) \tau \sigma \tau = (y) \sigma \tau = (\hat{y})\tau = y = (x)\tau. \]
Therefore, $\tau$ is regular. 
\end{proof}

\begin{corollary}\label{cor-regular}
Let $G$ be a nontrivial group and $A$ a set with at least two elements. Let $\tau \in \CA(G;A)$, and suppose there is a constant configuration $\mathbf{k} \in (A^G)\tau$ such that there is no constant configuration of $A^G$ mapped to $\mathbf{k}$ under $\tau$. Then $\tau$ is not regular.
\end{corollary}
\begin{proof}
The result follows by Remark \ref{constant} and Lemma \ref{le-regular}.  
\end{proof}

In the following examples we see how Corollary \ref{cor-regular} may be used to show that some well-known CA are not regular.

\begin{example}
Let $\phi \in \CA(\mathbb{Z}; \{0,1 \})$ be the Rule 110 elementary CA, and consider the constant configuration $\mathbf{1}$. Define $x := \dots 10101010 \dots \in \{0,1 \}^{\mathbb{Z}}$, and note that $(x)\phi = \mathbf{1}$. Since $(\mathbf{1})\phi = \mathbf{0}$ and $(\mathbf{0})\phi= \mathbf{0}$, Corollary \ref{cor-regular} implies that $\phi$ is not regular. 
\end{example}

\begin{example}
Let $\tau \in \CA(\mathbb{Z}^2; \{0,1 \})$ be Conway's Game of Life, and consider the constant configuration $\mathbf{1}$ (all cells alive). By \cite[Exercise 1.7.]{CSC10}, $\mathbf{1}$ is in the image of $\tau$; since $(\mathbf{1})\tau = \mathbf{0}$ (all cells die from overpopulation) and $(\mathbf{0})\tau = \mathbf{0}$, Corollary \ref{cor-regular} implies that $\tau$ is not regular.
\end{example}

The following theorem applies to CA over arbitrary groups and sets, and it shows that, except for the trivial cases, $\CA(G;A)$ always contains non-regular elements.

\begin{theorem} \label{th:regular}
Let $G$ be a group and $A$ a set. The semigroup $\CA(G;A)$ is regular if and only if $\vert G \vert = 1$ or $\vert A \vert = 1$.
\end{theorem}
\begin{proof}
If $\vert G \vert = 1$ or $\vert A \vert = 1$, then $\CA(G;A) = \Tran(A)$ or $\CA(G;A)$ is the trivial semigroup with one element, respectively. In both cases, $\CA(G;A)$ is regular (see \cite[Exercise 1.9.1]{CP61}).

Assume that $\vert G \vert \geq 2$ and $\vert A \vert \geq 2$. Suppose that $\{ 0,1\} \subseteq A$. Let $S := \{e,g,g^{-1}\} \subseteq G$, where $e$ is the identity of $G$ and $e \neq g \in G$ (we do not require $g \neq g^{-1}$). For $i =1,2$, let $\tau_i \in \CA(G;A)$ be the cellular automaton defined by the local function $\mu_i : A^S \to A$, where, for any $x \in A^S$,
\begin{align*}
(x)\mu_1   &:=  \begin{cases}
(e)x & \text{if } (e)x = (g)x = (g^{-1})x , \\
0 & \text{otherwise};
\end{cases}  \\
(x) \mu_2 &:= \begin{cases}
1 & \text{if } (e)x = (g)x = (g^{-1})x= 0 , \\
(e)x & \text{otherwise}.
\end{cases} 
\end{align*}
We shall show that $\tau := \tau_2 \tau_1 \in \CA(G;A)$ is not regular.

Consider the constant configurations $\mathbf{0}, \mathbf{1} \in A^G$. Let $z \in A^G$ be defined by
\[ (h)z := \begin{cases}
m \mod(2) & \text{if } h=g^m, m \in \mathbb{N} \text{ minimal}, \\
0 & \text{ otherwise}. 
\end{cases} \]   

\begin{figure}[h]
\centering
\begin{tikzpicture}[vertex/.style={circle, draw, fill=none, inner sep=0.2cm}]
    \vertex{1}{1}{2}    \node at (1,3.4) {$z$};  
    \vertex{2}{3}{1}    \node at (3,1.7) {$\mathbf{k}$};   
   \vertex{3}{1}{1}    \node at (1,1.7) {$\mathbf{0}$};  
   \vertex{4}{1}{0}   \node at (1,0) {$\mathbf{1}$};   

	\path[->,every loop/.style={min distance=6mm,in=65,out=120,looseness=5}] (1) edge [loop above] node[pos=.5,above] {\small{$\tau_2$}} () ;
	\path[->,every loop/.style={min distance=6mm,in=65,out=120,looseness=5}] (2) edge [loop above] node[pos=.4,above] {\small{\ \ $\tau_1,\tau_2$}} () ;
	\path[->,every loop/.style={min distance=6mm,in=210,out=155,looseness=5}] (3) edge [loop above] node[pos=.5,left] {\small{$\tau_1$}} () ;
	\path[->,every loop/.style={min distance=6mm,in=-30,out=25,looseness=5}] (4) edge [loop above] node[pos=.5,right] {\small{$\tau_1, \tau_2$}} () ;

	\draw[->] (1) -- (3) node[pos=.5,right] {\small{$\tau_1$}}; 
	\draw[->] (3) -- (4) node[pos=.5,right] {\small{$\tau_2$}};

   \end{tikzpicture} 
\caption{Images of $\tau_1$ and $\tau_2$.}
\label{Fig1}
   \end{figure}
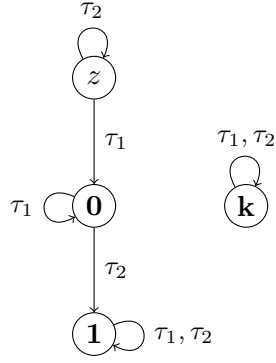

Figure \ref{Fig1} illustrates the images $z$, $\textbf{0}$, $\textbf{1}$, and $\mathbf{k} \neq \textbf{0}, \textbf{1}$ (in case it exists) under $\tau_1$ and $\tau_2$. Clearly,
\[ (\mathbf{0})\tau = ( \mathbf{0})\tau_2 \tau_1 = (\mathbf{1})\tau_1 = \mathbf{1}. \]
In fact, 
 \[ (\mathbf{k})\tau = \begin{cases}
 \mathbf{1} & \text{if } \mathbf{k}=\mathbf{0},  \\
 \mathbf{k} & \text{otherwise}. 
  \end{cases} \] 
Furthermore,  
\[ (z)\tau = (z) \tau_2 \tau_1 = (z)\tau_1 = \textbf{0}. \]
Hence, $\mathbf{0}$ is a constant configuration in the image of $\tau$ but with no preimage among the constant configurations. By Corollary \ref{cor-regular}, $\tau$ is not regular. 
\end{proof}

Now that we know that $\CA(G;A)$ always contains both regular and non-regular elements (when $\vert G \vert \geq 2$ and $\vert A \vert \geq 2$), an interesting problem is to find a criterion that describes all regular CA. In the following sections, we solve this problem by adding some extra assumptions, such as finiteness and linearity. 


\section{Regular finite cellular automata} \label{finite}

In this section we characterise the regular elements in the monoid $\CA(G;A)$ when $G$ and $A$ are both finite (Theorem \ref{characterisation}). In order to achieve this, we summarise some of the notation and results obtained in \cite{CRG16a,CRG16b,CRG16c}.

\begin{definition} The following definitions apply for an arbitrary group $G$ and an arbitrary set $A$:
\begin{enumerate}
\item For any $x \in A^G$, the \emph{$G$-orbit} of $x$ in $A^G$ is $xG := \{ x \cdot g : g \in G  \}$.
\item For any $x \in A^G$, the \emph{stabiliser} of $x$ in $G$ is $G_x := \{ g \in G : x \cdot g = x \}$.
 \item A \emph{subshift} of $A^G$ is a subset $X \subseteq A^G$ that is \emph{$G$-invariant}, i.e. for all $x \in X$, $g \in G$, we have $x \cdot g \in X$, and closed in the prodiscrete topology of $A^G$. 
\item The \emph{group of  invertible cellular automata} over $G$ and $A$ is
\[ \ICA(G;A) := \{ \tau \in \CA(G;A) : \exists \phi \in \CA(G;A) \text{ such that } \tau \phi = \phi \tau = \id \}. \]
\end{enumerate}
\end{definition}

In the case when $G$ and $A$ are both finite, every subset of $A^G$ is closed in the prodiscrete topology, so the subshifts of $A^G$ are simply unions of $G$-orbits. Moreover, as every map $\tau : A^G \to A^G$ is continuous in this case, $\CA(G;A)$ consists of all the $G$-equivariant maps of $A^G$. Theorem \ref{conjugate} is easily deduced from Lemmas 3, 9 and 10 in \cite{CRG16c}.

If $M$ is a group, or a monoid, write $K \leq M$ if $K$ is a subgroup, or a submonoid, of $M$, respectively.

\begin{theorem}\label{conjugate}
Let $G$ be a finite group of size $n \geq 2$ and $A$ a finite set of size $q \geq 2$. Let $x,y \in A^G$.
\begin{description}
\item[(i)] Let $\tau \in \CA(G;A)$. If $(x)\tau \in (xG)$, then $\tau \vert_{xG} \in \Sym(xG)$.
\item[(ii)] There exists $\tau \in \ICA(G;A)$ such that $(x)\tau = y$ if and only if $G_x = G_y$. 
\item[(iii)] There exists $\tau \in \CA(G;A)$ such that $(x) \tau = y$ if and only if $G_x \leq G_y$.
 \end{description} 
\end{theorem}

\begin{theorem}\label{characterisation}
Let $G$ be a finite group and $A$ a finite set of size $q\geq 2$. Let $\tau \in \CA(G;A)$. Then, $\tau$ is regular if and only if for every $y \in (A^G)\tau$ there is $x \in A^G$ such that $(x)\tau = y$ and $G_x = G_y$.
\end{theorem}
\begin{proof}
First, suppose that $\tau$ is regular. By Lemma \ref{le-regular}, there exists $\phi \in \CA(G;A)$ such that for every $y \in (A^G) \tau$ there is $\hat{y} \in A^G$ with $(\hat{y})\tau = y$ and $(y)\phi = \hat{y}$. Take $x := \hat{y}$. By Theorem \ref{conjugate}, $G_x \leq G_y$ and $G_y \leq G_x$. Therefore, $G_x = G_y$. 

Conversely, suppose that for every $y \in (A^G)\tau$ there is $x \in A^G$ such that $(x)\tau = y$ and $G_x = G_y$. Choose pairwise distinct $G$-orbits $y_1G, \dots, y_\ell G$ such that  
\[ (A^G)\tau = \bigcup_{i=1}^{\ell} y_i G. \]
For each $i$, fix $y_i^\prime \in A^G$ such that $(y_i^\prime) \tau =  y_i$ and $G_{y_i}= G_{y_i^\prime}$. We define $\phi : A^G \to A^G$ as follows: for any $z \in A^G$,
\[ (z)\phi := \begin{cases}
z & \text{if } z \not \in (A^G)\tau, \\
y_i^\prime \cdot g & \text{if } z = y_i \cdot g \in y_i G.
 \end{cases} \]
The map $\phi$ is well-defined because
\[ y_i \cdot g = y_i \cdot h \ \Longleftrightarrow \ hg^{-1} \in G_{y_i} = G_{y_i^\prime} \ \Longleftrightarrow  \ y_i^\prime \cdot g = y_i^\prime \cdot h.  \]
Clearly, $\phi$ is $G$-equivariant, so $\phi \in \CA(G;A)$. Now, for any $x \in A^G$ with $(x)\tau = y_i \cdot g$, 
\[ (x) \tau \phi \tau = (y_i \cdot g) \phi \tau = (y_i^\prime \cdot g) \tau = (y_i^\prime) \tau \cdot g =  y_i \cdot g = (x) \tau.  \]
This proves that $\tau \phi \tau = \tau$, so $\tau$ is regular. 
\end{proof}

Our goal now is to find a regular submonoid of $\CA(G;A)$ and describe its structure (see Theorem \ref{le:max regular}). In order to achieve this, we need some further terminology and basic results. 

Say that two subgroups $H_1$ and $H_2$ of $G$ are \emph{conjugate} in $G$ if there exists $g \in G$ such that $g^{-1} H_1 g = H_2$. This defines an equivalence relation on the subgroups of $G$. Denote by $[H]$ the conjugacy class of $H \leq G$.
Define the \emph{box} in $A^G$ corresponding to $[H]$, where $H \leq G$, by 
\[ B_{[H]}(G;A) := \{ x \in A^G : [G_x] =  [H] \}. \]
As any subgroup of $G$ is the stabiliser of some configuration in $A^G$, the set $\{ B_{[H]}(G;A) : H \leq G \}$ is a partition of $A^G$. Note that $B_{[H]}(G;A)$ is a subshift of $A^G$ (because $G_{(x \cdot g)} = g^{-1} G_x g$) and, by the Orbit-Stabiliser Theorem, all the $G$-orbits contained in $B_{[H]}(G;A)$ have equal sizes. When $G$ and $A$ are clear from the context, we write simply $B_{[H]}$ instead of $B_{[H]}(G;A)$. 

\begin{example}
For any finite group $G$ and finite set $A$ of size $q$, we have 
\[ B_{[G]} = \{ \mathbf{k} \in A^G : \mathbf{k} \text{ is constant} \}. \]
\end{example}

For any subshift $C \subseteq A^G$, define
\[ \CA(C) := \{ \tau \in \Tran(C) : \tau \text{ is $G$-equivariant} \}. \]  
In particular, $\CA(A^G) = \CA(G; A)$. Clearly, 
\[ \CA(C) = \{ \tau \vert_C : \tau \in \CA(G;A), \tau(C) \subseteq C \}. \]

A submonoid $R \leq M$ is called \emph{maximal regular} if there is no regular monoid $K$ such that $R < K < M$.  

\begin{theorem} \label{le:max regular}
Let $G$ be a finite group and $A$ a finite set of size $q\geq 2$. Let
\[ R := \left\{ \sigma \in \CA(G;A) : G_x = G_{(x)\sigma} \text{ for all } x \in A^G  \right\}. \]
\begin{description}
\item[(i)] $\ICA(G;A) \leq R$.
\item[(ii)] $R$ is a regular monoid. 
\item[(iii)] $R \cong \prod_{H \leq G} \CA(B_{[H]})$.
\item[(iv)] $R$ is not a maximal regular submonoid of $\CA(G;A)$.
\end{description}
\end{theorem}
\begin{proof}
Part \textbf{(i)} and \textbf{(iii)} are trivial while part \textbf{(ii)} follows by Theorem \ref{characterisation}.

For part \textbf{(iv)}, let $x, y \in A^G$ be such that $G_x < G_y$, so $x$ and $y$ are in different boxes. Define $\tau \in \CA(G;A)$ such that $(x) \tau = y$, $(B_{[G_y]})\tau = yG$, and $\tau$ fixes any other configuration in $A^G \setminus (B_{[G_y]} \cup \{ xG\})$. It is clear by Theorem \ref{characterisation} that $\tau$ is regular. We will show that $K:=\langle R, \tau \rangle$ is a regular submonoid of $\CA(G;A)$. Let $\sigma \in K$ and $z \in (A^G)\sigma$. If $\sigma \in R$, then it is obviously regular, so assume that $\sigma = \rho_1 \tau \rho_2$ with $\rho_1 \in K$ and $\rho_2 \in R$. If $z \in A^G \setminus (B_{[G_y]})$, it is clear that $z$ has a preimage in its own box; otherwise $(B_{[G_y]}) \sigma = (yG) \rho_2 = zG$ and $z$ has a preimage in $B_{[G_y]}$. Hence $\sigma$ is regular and so is $K$. 
\end{proof}


\section{Regular linear cellular automata} \label{linear}

Let $V$ a vector space over a field $\mathbb{F}$. For any group $G$, the configuration space $V^G$ is also a vector space over $\mathbb{F}$ equipped with the pointwise addition and scalar multiplication. Denote by $\End_{\mathbb{F}}(V^G)$ the set of all $\mathbb{F}$-linear transformations of the form $\tau : V^G \to V^G$. Define
\[ \LCA(G;V) := \CA(G;V) \cap \End_{\mathbb{F}}(V^G). \]
Note that $\LCA(G;V)$ is not only a monoid, but also an $\mathbb{F}$-algebra (i.e. a vector space over $\mathbb{F}$ equipped with a bilinear binary product), because, again, we may equip $\LCA(G;V)$ with the pointwise addition and scalar multiplication. In particular, $\LCA(G;V)$ is also a ring. 

As in the case of semigroups, von Neumann regular rings have been widely studied and many important results have been obtained. In this chapter, we study the regular elements of $\LCA(G;V)$ under some natural assumptions on the group $G$. 

First, we introduce some preliminary results and notation. The \emph{group ring} $R[G]$ is the set of all functions $f : G \to R$ with finite support (i.e. the set $\{ g \in G : (g)f \neq 0 \}$ is finite). Equivalently, the group ring $R[G]$ may be defined as the set of all formal finite sums $\sum_{g \in G} a_g g$ with $a_g \in R$. The multiplication in $R[G]$ is defined naturally using the multiplications of $G$ and $R$:
\[ \sum_{g \in G} a_g g \sum_{h \in G} a_h h = \sum_{g, h \in G} a_g a_h  gh. \]
 If we let $R :=\End_{\mathbb{F}}(V)$, it turns out that $\End_{\mathbb{F}}(V)[G]$ is isomorphic to $\LCA(G;V)$ as $\mathbb{F}$-algebras (see \cite[Theorem 8.5.2]{CSC10}). 

Define the \emph{order} of $g \in G$ by $o(g) :=  \vert \langle g \rangle \vert$ (i.e. the size of the subgroup generated by $g$). The group $G$ is \emph{torsion-free} if the identity is the only element of finite order; for instance, the groups $\mathbb{Z}^d$, for $d \in \mathbb{N}$, are torsion-free groups. The group $G$ is \emph{elementary amenable} if it may be obtained from finte groups or abelian groups by a sequence of group extensions or direct unions.   

In the following theorem we characterise the regular linear cellular automata over fields and torsion-free elementary amenable groups (such as $\mathbb{Z}^d$, $d \in \mathbb{N}$).

\begin{theorem}\label{torsion-free}
Let $G$ be a torsion-free elementary amenable group and let $V = \mathbb{F}$ be any field. A non-zero element $\tau \in \LCA(G; \mathbb{F})$ is regular if and only if it is invertible.
\end{theorem}
\begin{proof}
It is clear that any invertible element is regular. Let $\tau \in \LCA(G;\mathbb{F})$ be non-zero regular. In this case, $\End_{\mathbb{F}}(\mathbb{F}) \cong \mathbb{F}$, so $\LCA(G;\mathbb{F}) \cong \mathbb{F}[G]$. By definition, there exists $\sigma \in \LCA(G;V)$ such that $\tau \sigma \tau = \tau$. As $\LCA(G;V)$ is a ring, the previous is equivalent to
\[ \tau (\sigma\tau - 1) = 0, \]
where $1 = 1e$ and $0=0e$ are the identity and zero endomorphisms, respectively. Since $\tau \neq 0$, either $\sigma\tau - 1 = 0$, in which case $\tau$ is invertible, or $\sigma \tau - 1$ is a zero-divisor. However, it was established in \cite[Theorem 1.4]{KLM88} that $\mathbb{F}[G]$ has no zero-divisors whenever $G$ is a torsion-free elementary amenable group. Hence, $\tau$ is invertible.
\end{proof}

The argument of the previous result works as long as the group ring $\mathbb{F}[G]$ has no zero-divisor. This is connected with the well-known Kaplansky's conjecture which states that $\mathbb{F}[G]$ has no zero-divisors when $G$ is a torsion-free group.    

The \emph{characteristic} of a field $\mathbb{F}$, denoted by $\Char(\mathbb{F})$, is the smallest $k \in \mathbb{N}$ such that 
\[ \underbrace{1+1+\cdots+1}_{k \text{ times}} = 0 ,\]
where $1$ is the multiplicative identity of $\mathbb{F}$. If no such $k$ exists we say that $\mathbb{F}$ has characteristic $0$.

A group $G$ is \emph{locally finite} if every finitely generated subgroup of $G$ is finite; in particular, the order of every element of $G$ is finite. Examples of such groups are finite groups and infinite direct sums of finite groups.

\begin{theorem}\label{regular-ring}
Let $G$ be a group and let $V$ be a finite-dimensional vector space over $\mathbb{F}$. Then, $\LCA(G;V)$ is regular if and only if $G$ is locally finite and $\Char(\mathbb{F}) \nmid o(g)$, for all $g \in G$.
\end{theorem}
\begin{proof}
By \cite[Theorem 3]{C63} (see also \cite{A,ML}), we have that a group ring $R[G]$ is regular if and only if $R$ is regular, $G$ is locally finite and $o(g)$ is a unit in $R$ for all $g \in G$. In the case of $\LCA(G;V) \cong \End_{\mathbb{F}}(V)[G]$, since $\dim(V) := n < \infty$, the ring $R:=\End_{\mathbb{F}}(V) \cong M_{n\times n}(\mathbb{F})$ is regular (see \cite[Theorem 1.7]{G79}. The condition that $o(g)$, seen as the matrix $o(g) I_n$, is a unit in $M_{n\times n}(\mathbb{F})$ is satisfied if and only if $o(g)$ is nonzero in $\mathbb{\mathbb{F}}$, which is equivalent to $\Char(\mathbb{F}) \nmid o(g)$, for all $g \in G$. 
\end{proof}

\begin{corollary}
Let $G$ be a group and let $V$ be a finite-dimensional vector space over a field $\mathbb{F}$ of characteristic $0$. Then, $\LCA(G;V)$ is regular if and only if $G$ is locally finite. 
\end{corollary}

Henceforth, we focus on the regular elements of $\LCA(G;V)$ when $V$ is a one-dimensional vector space (i.e. $V$ is just the field $\mathbb{F}$). In this case, $\End_{\mathbb{F}}(\mathbb{F}) \cong \mathbb{F}$, so $\LCA(G;\mathbb{F})$ and $\mathbb{F}[G]$ are isomorphic as $\mathbb{F}$-algebras. 

A non-zero element $a$ of a ring $R$ is called \emph{nilpotent} if there exists $n > 0$ such that $a^n = 0$. The following basic result will be quite useful in the rest of this section.

\begin{lemma}\label{nilpotent}
Let $R$ be a commutative ring. If $a \in R$ is nilpotent, then $a$ is not a regular element.
\end{lemma}
\begin{proof}
Let $R$ be a commutative ring and $a \in R$ a nilpotent element. Let $n >0$ be the smallest integer such that $a^n = 0$. Suppose $a$ is a regular element, so there is $x \in R$ such that $axa=a$. By commutativity, we have $a^2 x = a$. Multiplying both sides of this equation by $a^{n-2}$ we obtain $0 = a^{n}  x = a^{n-1}$, which contradicts the minimality of $n$. 
\end{proof}

\begin{example}
Suppose that $G$ is a finite abelian group and let $\mathbb{F}$ be a field such that $\Char(\mathbb{F}) \mid \vert G \vert$. By Theorem \ref{regular-ring}, $\LCA(G;\mathbb{F})$ must have elements that are not regular. For example, let $s := \sum_{g \in G} g \in \mathbb{F}[G]$. As $sg = s$, for all $g \in G$, and $\Char(\mathbb{F}) \mid \vert G \vert$, we have $s^2 = \vert G \vert s = 0$. Clearly, $\mathbb{F}[G]$ is commutative because $G$ is abelian, so, by Lemma \ref{nilpotent}, $s$ is not a regular element. 
\end{example}

We finish this section with the special case when $G$ is the cyclic group $\mathbb{Z}_n$ and $\mathbb{F}$ is a finite field with $\Char(\mathbb{F}) \mid n$. By Theorem \ref{regular-ring}, not all the elements of $\LCA(\mathbb{Z}_n; \mathbb{F})$ are regular, so how many of them are there? In order to count them we need a few technical results about commutative rings.

An \emph{ideal} $I$ of a commutative ring $R$ is a subring such that $rb  \in I$ for all $r \in R$, $b \in I$. For any $a \in R$, the \emph{principal ideal} generated by $a$ is the ideal $\langle a \rangle := \{ ra : r \in R \}$. A ring is called \emph{local} if it has a unique maximal ideal. 

Denote by $\mathbb{F}[x]$ the ring of polynomials with coefficients in $\mathbb{F}$. When $G \cong \mathbb{Z}_n$, we have the following isomorphisms as $\mathbb{F}$-algebras:
\[ \LCA(\mathbb{Z}_n;\mathbb{F}) \cong \mathbb{F}[\mathbb{Z}_n] \cong \mathbb{F}[x] / \langle x^n -1 \rangle, \]
where $\langle x^n -1 \rangle$ is a principal ideal in $\mathbb{F}[x]$.

\begin{theorem} 
Let $n \geq 2$ be an integer, and let $\mathbb{F}$ be a finite field of size $q$ such that $\Char(\mathbb{F}) \mid n$. Consider the following factorization of $x^n - 1$ into irreducible elements of $\mathbb{F}[x]$: 
\[ x^n -1 = p_1(x)^{m_1} p_2(x)^{m_2} \dots p_r(x)^{m_r}. \]
For each $i = 1, \dots r$, let $d_i := \deg(p_i(x))$. Then, the number of regular elements in $\LCA(\mathbb{Z}_n; \mathbb{F})$ is exactly
\[ \prod_{i=1}^{r} \left( ( q^{d_i} - 1) q^{d_i (m_i-1)} +1 \right). \]
\end{theorem}
\begin{proof}
Recall that
\[ \LCA(\mathbb{Z}_n;\mathbb{F}) \cong \mathbb{F}[x] / \langle x^n -1 \rangle. \]
By the Chinese Remainder Theorem,
\[ \mathbb{F}[x] / \langle x^n - 1 \rangle \cong \mathbb{F}[x]/ \langle p_1(x)^{m_1} \rangle \times \mathbb{F}[x]/ \langle p_2(x)^{m_2} \rangle \times \dots \times \mathbb{F}[x]/ \langle p_r(x)^{m_r} \rangle. \]
An element $a = (a_1, \dots, a_r)$ in the right-hand side of the above isomorphism is a regular element if and only if $a_i$ is a regular element in $\mathbb{F}[x]/ \langle p_i(x)^{m_i} \rangle$ for all $i = 1, \dots,r$. 

Fix $m := m_i$, $p(x) = p_i(x)$, and $d:=d_i$. Consider the principal ideals $A := \langle p(x) \rangle$ and $B := \langle p(x)^m \rangle$ in $\mathbb{F}[x]$. Then, $\mathbb{F}[x]/B$ is a local ring with unique maximal ideal $A/B$, and each of its nonzero elements is either nilpotent or a unit (i.e. invertible): in particular, the set of units of $\mathbb{F}[x]/B$ is precisely $(\mathbb{F}[x]/B) - (A/B)$. By the Third Isomorphism Theorem, $(\mathbb{F}[x]/B)/(A/B) \cong (\mathbb{F}[x]/A)$, so 
\[ \vert A/B \vert = \frac{\vert \mathbb{F}[x]/B \vert}{\vert \mathbb{F}[x]/A\vert} = \frac{q^{dm}}{q^d} = q^{d(m-1)}. \]
Thus, the number of units in $\mathbb{F}[x]/B$ is 
\[ \vert (\mathbb{F}[x]/B) - (A/B) \vert = q^{dm} - q^{d(m-1)} = (q^d - 1)q^{d(m-1)}. \] 

As nilpotent elements are not regular by Lemma \ref{nilpotent}, every regular element of $\mathbb{F}[x]/ \langle p_i(x)^{m_i} \rangle$ is zero or a unit. Thus, the number of regular elements in $\mathbb{F}[x]/ \langle p_i(x)^{m_i} \rangle$ is $( q^{d_i} - 1) q^{d_i (m_i-1) } +1$.  
\end{proof}

\subsubsection*{Acknowledgments.} We thank the referees of this paper for their insightful suggestions and corrections. In particular, we thank the first referee for suggesting the references \cite{A,C63,ML}, which greatly improved the results of Section \ref{linear}.

\Addresses

\end{document}